\documentclass{amsart}
\usepackage{amsmath, amsthm}
\usepackage[foot]{amsaddr}
\usepackage{mathtools,hyperref}
\usepackage{amscd,amssymb,amsfonts}
\usepackage[mathscr]{euscript}
\usepackage{enumerate}
\usepackage{stackengine}
\usepackage{multicol}
\usepackage{subcaption}
\usepackage{dsfont}
\DeclareMathOperator{\conv}{conv}
\DeclareMathOperator{\M}{M}
\DeclareMathOperator{\FM}{FM}
\DeclareMathOperator{\K}{K}
\DeclareMathOperator{\PP}{P}
\DeclareMathOperator{\NP}{NP}

\DeclareMathOperator{\x}{\mathbf{x}}

\theoremstyle{plain}
	\newtheorem{theorem}{Theorem}[section]
	\newtheorem{lemma}[theorem]{Lemma}
	\newtheorem{proposition}[theorem]{Proposition}

        \theoremstyle{definition}
	\newtheorem{definition}[theorem]{Definition}
	
	\newtheorem*{observation}{Observation}

\begin{document}
\title{Normality of $k$-Matching Polytopes of Bipartite Graphs}

        \author[Juan Camilo Torres]{Juan Camilo Torres$^1$}
\date{}
\address{$^1$ Departamento de Matem\'{a}ticas, Universidad de los Andes, Bogot\'{a}, Colombia}
\address{Departamento de Matem\'{a}ticas, Universidad Militar Nueva Granada, Cajic\'{a}, Colombia}
\thanks{Email: jc.torresc@uniandes.edu.co}
\thanks{The author was supported by an internal research grant (INV-2018-48-1373) from the Faculty of Sciences of the Universidad de los Andes. The author was also supported in his doctoral studies, of which this project forms a part, by the Colombian Government through  Minciencias.}

\begin{abstract}
The $k$-matching polytope of a graph is the convex hull of all its matchings of a given size $k$ when they are considered as indicator vectors. In this paper, we prove that the $k$-matching polytope of a bipartite graph is normal, that is, every integer point in its $t$-dilate is the sum of $t$ integers points of the original polytope. This generalizes the known fact that Birkhoff polytopes are normal. As a preliminary result, we prove that for bipartite graphs the $k$-matching polytope is equal to the fractional $k$-matching polytope, having thus the $H$-representation of the polytope. This generalizes the Birkhoff-Von Neumann Theorem which establish that every doubly stochastic matrix can be written as a convex combination of permutation matrices.   
\end{abstract}

\maketitle

\section{Introduction}
Matchings are important in combinatorial optimization not only for their applications, but also because of their intriguing status from the viewpoint of computational complexity. There is a polynomial-time algorithm to find a perfect matching if one exists, or to show that one does not exist. This is the method of augmenting paths, which is one of the first non-trivial polynomial-time algorithms that we encounter in combinatorial optimization. Furthermore, although decision problems such as \textquotedblleft does $G$ have a perfect matching?", or \textquotedblleft does $G$ have a $k$-matching?" are in $\PP$, the counting problem \textquotedblleft how many perfect matchings does $G$ have?" is $\#\PP$-complete, even for bipartite graphs (see, for example, \cite{DK11}). That means that counting matchings is just as hard as counting 3-colorings or large cliques or other objects for which the decision problem is $\NP$-complete.

After indexing the vertices of a graph, matchings can be seen as vectors, and we can take any subset of matchings and form the corresponding convex hull to obtain a polytope. The polytopes of all matchings and of maximal matchings are well-studied but much less so when we only consider matchings of a given size $k$. We will call them $k$-matching polytopes, and they will be the focus of our paper. (There is a concept similar in name called the $\mathbf{b}$-matching polytope but it is a different object. For a graph $G=(V,E)$ and a vector $\mathbf{b}\in\mathbb{Z}^V$, the $\mathbf{b}$-matching polytope is the convex hull of all vectors $\mathbf{x}\in\mathbb{Z}^E$ with nonnegative entries such that $\sum_{e\in v}x_e\leq b_v$ for all $v\in V$. See \cite{B13} for more about these of polytopes.)

The main purpose of our paper is to prove that the $k$-matching polytope of a bipartite graph is normal, that is, every integer point in its $t$-dilate is the sum of $t$ integer points of the original polytope, which will be done in Section \ref{sec:normal}. It is known that Birkhoff polytopes are normal which is a special case of our result. As a preliminary result, in Section \ref{sec:hrepr}, we will prove that the $k$-matching polytope is equal to the fractional $k$-matching polytope for bipartite graphs, thus obtaining the $H$-representation for this polytope  (a result proved in \cite{GL09} for the special case of $k$-assignment polytopes). These give a generalization of the Birkhoff-Von Neumann Theorem that asserts that every doubly stochastic matrix is a convex combination of permutation matrices.

The results presented in this paper were part of the doctoral thesis of the author, see \cite{T21}. He would like to thank his advisor Tristram Bogart for useful discussions on this topic.
\section{Preliminaries}
In this section, we recall some basic concepts and results about matchings, which have been of special interest in graph theory and combinatorial optimization.
\begin{definition}
A \emph{matching} $M$ of a graph $G=(V,E)$ is a collection of edges of $G$ such that no two different edges in $M$ are incident to a common vertex. If $|M|=k$, we say that $M$ is a $k$-\emph{matching}, and the matching is \emph{perfect} if every vertex of $G$ belongs to an edge in $M$. The vector $\chi(M)\in\mathbb{R}^E$ defined by
\[\chi(M)_e=
\begin{cases}
1 & \text{if } e\in M\\
0 & \text{if } e\notin M
\end{cases}
\]
is called the characteristic vector of $M$.
\end{definition}
Given a graph, we can construct a polytope that encodes its matching structure.
\begin{definition}
Let $G$ be a graph. The polytope
\[\M(G):=\conv\{\chi(M): M \text{ is a matching of } G\}\]
in $\mathbb{R}^E$ is called the \emph{matching polytope} of $G$.
\end{definition}
Associated to this construction is the fractional matching polytope, which is a natural linear relaxation of the matching polytope.
\begin{definition}
Let $G$ be a graph. The polytope
\[\FM(G):= \{ \x\in\mathbb{R}^E:\sum_{e\ni v}x_e\leq 1\  \ \forall v\in V,\ x_e\geq 0\  \ \forall e\in E\}\]
in $\mathbb{R}^E$ is called the \emph{fractional matching polytope} of $G$.  
\end{definition}
In general the matching polytope and the fractional matching polytope are not equal. Nonetheless, they are equal for bipartite graphs.
\begin{theorem}\label{fractional}
If $G$ is a bipartite graph, then $\M(G)=\FM(G)$.
\end{theorem}
So for bipartite graphs, we know the $V$-representation as well as a simple $H$-representation for the matching polytope. Another fundamental result for matchings in bipartite graphs is Hall's Marriage Theorem.
\begin{theorem}[Hall's Marriage Theorem]
Let $G$ be a bipartite graph on $A\sqcup B$. Then $G$ has a matching that covers $A$ if and only if for every $S\subseteq A$, the set of neighbors of $S$ (the set of vertices that are connected with at least one vertex in $S$), denoted by $\Gamma(S)$, satisfies $|\Gamma(S)|\geq|S|$.
\end{theorem}
 For a thorough reference on matchings (including the above results), see the book \cite{LP09}.  Also, we will need the following well known result about the union of two matchings.
\begin{lemma}\label{l2}
Let $G$ be a bipartite graph on $A\sqcup B$, $V_1\subseteq A$, and $V_2\subseteq B$. If $M_1$ is a matching that covers $V_1$ and $M_2$ a matching that covers $V_2$, then $M_1\cup M_2$ contains a matching that covers $V_1\cup V_2$.
\end{lemma}
\begin{proof}
Consider the digraph $D$ induced by $M_1\cup M_2$ where each edge in $M_1$ induces a directed edge from $A$ to $B$, and each edge in $M_2$ induces a directed edge from $B$ to $A$. Every vertex in $D$ has indegree and outdegree of at most 1, so $D$ is composed by disjoint directed even cycles and paths starting at a vertex in $V_1\cup V_2$ and ending at a vertex not in $V_1\cup V_2$ (since every vertex in this set has an outgoing edge). From here, it is clear we can take a matching from $M_1\cup M_2$ that covers $V_1\cup V_2$.
\end{proof} 
We can also construct polytopes from matchings of a given size $k$, which are less well-studied for $k$ other than the maximum possible size.
\begin{definition}
Let $G$ be a graph and $k\in\mathbb{N}$.  The polytope
\[\M_k(G):=\conv\{\chi(M): M \text{ is a $k$-matching of } G\}\]
in $\mathbb{R}^E$ is called the $k$-\emph{matching polytope} of $G$, and the polytope
\[\FM_k(G):= \{ \x\in\mathbb{R}^E:\sum_{e\ni v}x_e\leq 1\  \ \forall v\in V,\ x_e\geq 0\  \ \forall e\in E,\ \sum_{e\in E}x_e=k\}\]
is called the \emph{fractional} $k$-\emph{matching polytope} of $G$.  
\end{definition}
We will see in the next section that again these two polytopes are equal for bipartite graphs. Even further, we can work with matchings up to a given size and define analogously the polytopes $\M_{\leq k}(G)$ and $\FM_{\leq k}(G)$ and still obtain the same equality for bipartite graphs.

We have defined the polytope $\M_k(G)$ for a bipartite graph $G=(V,E)$ on $A\sqcup B$ as living in $\mathbb{R}^E$, but we can also think of it as a polytope in $\mathbb{R}^{m\times n}$, where $m$ and $n$ are the sizes of $A$ and $B$ respectively, in the following way:
if $A=\{a_1,\ldots,a_m\}$ and $B=\{b_1,\ldots,b_n\}$, then a matching $M$ of $G$ can be represented as a 0/1 matrix of size $m\times n$ with a one in those entries $(i,j)$ such that $\{a_i,b_j\}\in M$ and zeros elsewhere. In this case $\M_k(G)=\conv\{M\in\mathbb{R}^{m\times n}:M \text{ is a $k$-matching of } G\}$, and $\FM(G)$ is the set of matrices $X\in\mathbb{R}^{m\times n}$ such that the entries in each row and column sum to at most one, $X_{ij}=0$ when $\{a_i,b_j\}\notin E$, and all the entries of $X$ sum to $k$.

Of special interest is when $G=\K_{n,n}$. With the matrix representation for matchings, the $n$-matchings (which are the perfect matchings) of $K_{n,n}$ correspond to permutation matrices.
\begin{definition}
Let $n\in\mathbb{N}$. The polytope
\begin{align*}
B_n:=\lbrace N\in\mathbb{R}^{n\times n}: & \sum_{j=1}^n x_{ij} = 1\ \ \forall\ 1\leq i\leq n,\\
& \sum_{i=1}^n x_{ij} = 1\ \ \forall\ 1\leq j\leq n,\\
& x_{ij}\geq 0\ \ \forall\ 1\leq i,j\leq n\rbrace
\end{align*}
is called a Birkhoff polytope. An element of $B_n$ is called a $n\times n$ doubly stochastic matrix.
\end{definition}
We can see that $B_n=\FM_n(K_{n,n})$, and the equality $\M_n(K_{n,n})=\FM_n(K_{n,n})$ is the famous Birkhoff-von Neumann Theorem which says that every doubly stochastic matrix is a convex combination of permutation matrices. Also the polytopes $\M_k(\K_{m,n})$ are studied in \cite{GL09} under the name of $k$-assignment polytopes.

Using the Birkhoff polytopes as example, we transition to the concept of normality.
\begin{definition}
A polytope $P$ is normal if for all $t\in\mathbb{N}$, every integer point (a point with integer coordinates) in $tP:=\{t\mathbf{x}:\mathbf{x}\in P\}$ is equal to the sum of $t$ integer points in $P$.
\end{definition}
It is known that Birkhoff polytopes are normal, and in Section \ref{sec:normal} we generalize that by showing that in general $k$-matching polytopes of bipartite graphs are normal. Normal polytopes have connections with monoid algebras \cite{BG09} and toric varities \cite{CLS11}. A recent survey on normal polytopes is \cite{G03}.
\section{The $H$-representation of the $k$-matching polytope of a bipartite graph}\label{sec:hrepr}
In this section we extend a standard result on matchings in bipartite graphs, $\M(G)=\FM(G)$, to $k$-matchings. That is, we will prove that $\M_k(G)=\FM_k(G)$ if $G$ is a bipartite graph. This result was known and proved in \cite{GL09} for the special case of $k$-assignment polytopes, that is, when $G=K_{m,n}$.
\begin{theorem}\label{h-representation}
Let $G=(V,E)$ be a bipartite graph and $k\in\mathbb{N}$. Then $\M_k(G)=\FM_k(G)$.
\end{theorem}
\begin{proof}
Clearly $\M_k(G)\subseteq\FM_k(G)$ and $\FM_k(G)\cap\mathbb{Z}^E\subseteq\M_k(G)$. So it is enough to prove that every vertex of $\FM_k(G)$ has integer coordinates. We do so by showing that if $\x$ is a non-integer point of $\FM_k(G)$, then it can be written as $\x=\frac{1}{2}(\x'+\x'')$ where $\x',\x''\in\FM_k(G)\setminus\{\x\}$, and thus it cannot be a vertex of $\FM_k(G)$. (We construct $\x'$ and $\x''$ by properly modifying $\x$; this is the same idea behind the proof $\M(G)=\FM(G)$, but now we need to be careful with the extra condition $\sum_{e\in E}x_e=k$.)

Let $\x$ be a point of $\FM_k(G)$ that is not integer, and consider the subgraph $H:=(V,F)$ where $F:=\{e\in E: x_e\notin\{0,1\}\}$. We divide the proof by cases.

\textbf{Case 1}: $H$ has a cycle.

If $H$ has a cycle, then it has to be an even cycle since $G$ is bipartite. Let $e_1,\ldots,e_m$ be the edges that appear in the cycle in that precise order.

Define $\varepsilon:=\min\{x_{e_1},\ldots,x_{e_m},1-x_{e_1},\ldots,1-x_{e_m}\}$. Let $\x'=(x'_e)_{e\in E}$ and $\x''=(x''_e)_{e\in E}$ where
\[x'_e=
\begin{cases}
x_e+\varepsilon & \text{if } e=e_i \text{ for some } i\in[m] \text{ odd},\\
x_e-\varepsilon & \text{if } e=e_i \text{ for some } i\in[m] \text{ even}, \tag{$\ast$}\\
x_e & \text{otherwise}, 
\end{cases} \]
and
\[x''_e=
\begin{cases}
x_e-\varepsilon & \text{if } e=e_i \text{ for some } i\in[m] \text{ odd},\\
x_e+\varepsilon & \text{if } e=e_i \text{ for some } i\in[m] \text{ even},\tag{$\ast\ast$}\\
x_e & \text{otherwise}. 
\end{cases} \]
Then $\x',\x''\in\FM_k(G)\setminus\{\x\}$ and $\x=\frac{1}{2}(\x'+\x'')$.

\textbf{Case 2}: $H$ has no cycles.

Let $P$ be a maximal path of $H$ with consecutive vertices $v_1,v_2,\ldots,v_m,v_{m+1}$ and edges $e_1=\{v_1,v_2\},\ldots,e_m=\{v_m,v_{m+1}\}$. Then $x_e=0$ if $e$ is an edge of $G$ incident to $v_1$ different from $e_1$. Indeed, it cannot be equal to 1 due to the inequality $\sum_{e\ni v_1}x_e\leq 1$, and it cannot be strictly between 0 and 1 since otherwise the path could be extended. A similar analysis can be done for $v_{m+1}$. Thus $\sum_{e\ni v_1}x_e=x_{e_1}<1$ and $\sum_{e\ni v_{m+1}}x_e=x_{e_m}<1$.

Furthermore, if $P$ is of even length, take $\varepsilon:=\min\{x_{e_1},\ldots,x_{e_m},1-x_{e_1},\ldots,1-x_{e_m}\}$. Define $\x'$ and $\x''$ as in ($\ast$) and ($\ast\ast$). Then $\x',\x''\in\FM_k(G)\setminus\{\x\}$ and $\x=\frac{1}{2}(\x'+\x'')$. In conclusion, if we can find a maximal path of even length in $H$ we are done.

\textbf{Subcase 2.1}: $H$ has no cycles, but it is connected.

If $P$ is of even length, we are done, so suppose $P$ has odd length. We analize the cases $H=P$ and $H\neq P$ separately.

Suppose $H=P$. If $\sum_{e\ni v}x_e=1$ for every interior vertex $v$ of $H$, then
\begin{align*}
k &= \sum_{e\in E}x_e\\
&= \sum_{e\in H}x_e + \text{ some integer}\\
&= (x_{e_1}+x_{e_2})+(x_{e_3}+x_{e_4})+\cdots+(x_{e_{m-2}}+x_{e_{m-1}})+x_{e_m}+ \text{ some integer}\\ 
&= 1+1+\cdots + 1 + x_{e_m}+ \text{ some integer}\\
&= x_{e_m}+ \text{ some integer}
\end{align*} 
which is a contradiction since $x_{e_m}$ is not an integer. So there is an interior vertex $v$ of $H$ such that $\sum_{e\ni v}x_e<1$. This vertex $v$ divides $H$ into two subpaths and one of them has to be of even length. Let's called this subpath $P_0$, and suppose without loss of generality that its consecutive edges are $e_1,\ldots,e_n$ ($n<m$). Take $\varepsilon:=\min\{x_{e_1},\ldots,x_{e_n},1-x_{e_1},\ldots,1-x_{e_{n-1}},1-x_{e_n}-x_{e_{n+1}}\}$. Now define $\x'$ and $\x''$ as in ($\ast$) and ($\ast\ast$) with $m$ replaced by $n$. Then $\x',\x''\in\FM_k(G)\setminus\{\x\}$ and $\x=\frac{1}{2}(\x'+\x'')$.

Now suppose $H\neq P$. Then there is an edge $e'_1$ that goes from an interior vertex $v_i$ of $P$ (it cannot be an end-vertex by the maximality of $P$) to a vertex of $H$ not belonging to $P$ (since there is no cycles in $H$). Extend the path $e_1,\ldots,e_{i-1},e'_1$ as much as possible to a maximal path $e_1,\ldots,e_{i-1},e'_1,\ldots,e'_j$. Then one of the maximal paths $e_1,\ldots,e_{i-1},e'_1,\ldots,e'_j$ or $e_m,\ldots,e_i,e'_1,\ldots,e'_j$ is of even length, and we are done.

\textbf{Subcase 2.2}: $H$ has no cycles and is not connected.
 
Finally, if $H$ is disconnected, then it has two maximal paths $P$ and $P'$ with no vertex in common. If one of them is of even length, we are done, so suppose both of them are of odd length. Let $e_1,\ldots,e_m$ be the consecutive edges of $P$ and $e'_1,\ldots,e'_t$ the consecutive edges of $P'$. Define $\varepsilon:=\min\{x_{e_1},\ldots,x_{e_m},x_{e'_1},\ldots,x_{e'_t},1-x_{e_1},\ldots,1-x_{e_m},1-x_{e'_1},\ldots,1-x_{e'_t}\}$. Let $\x'=(x'_e)_{e\in E}$ and $\x''=(x''_e)_{e\in E}$ where
\[x'_e=
\begin{cases}
x_e+\varepsilon & \text{if } e=e_i \text{ for some } i\in[m] \text{ odd}\\
x_e-\varepsilon & \text{if } e=e_i \text{ for some } i\in[m] \text{ even}\\
x_e-\varepsilon & \text{if } e=e'_i \text{ for some } i\in[t] \text{ odd}\\
x_e+\varepsilon & \text{if } e=e'_i \text{ for some } i\in[t] \text{ even}\\
x_e & \text{otherwise} 
\end{cases} \]
and
\[x''_e=
\begin{cases}
x_e-\varepsilon & \text{if } e=e_i \text{ for some } i\in[m] \text{ odd}\\
x_e+\varepsilon & \text{if } e=e_i \text{ for some } i\in[m] \text{ even}\\
x_e+\varepsilon & \text{if } e=e'_i \text{ for some } i\in[t] \text{ odd}\\
x_e-\varepsilon & \text{if } e=e'_i \text{ for some } i\in[t] \text{ even}\\
x_e & \text{otherwise} 
\end{cases}. \]
Then $\x',\x''\in\FM_k(G)\setminus\{\x\}$ and $\x=\frac{1}{2}(\x'+\x'')$.\\
\end{proof}
\begin{observation}
The same proof also works to show that $\M_{\leq k}(G)=\FM_{\leq k}(G)$ for bipartite graphs, where \[\M_{\leq k}(G):=\conv\{\chi(M): M \text{ is a matching of } G \text{ of size at most $k$}\}\]
and
\[\FM_{\leq k}(G):= \{ \x\in\mathbb{R}^E:\sum_{e\ni v}x_e\leq 1\  \ \forall v\in V,\ x_e\geq 0\  \ \forall e\in E,\ \sum_{e\in E}x_e\leq k\}.\]
A similar result holds for $\geq$ instead of $\leq$.
\end{observation}
\section{The normality of $k$-matching polytopes of bipartite graphs}\label{sec:normal}
The objective of this section is to prove that $k$-matching polytopes of bipartite graph are normal. As explained in the preliminaries, we will think of a matching in a bipartite graph not only as a set of edges but also as a matrix. Our starting point is the following result.
\begin{lemma}\label{lemmasubgraph}
If $G$ is a bipartite graph such that $\M_k(G)$ is normal, then so is $\M_k(H)$ for every subgraph $H$ of $G$.
\end{lemma}
\begin{proof}
By adding, if necessary, vertices with no incident edges, we can suppose without loss of generality that $H$ and $G$ have the same vertices. If $N$ is an integer point of $t\M_k(H)$, $N$ is also an integer point of $t\M_k(G)$. Since $\M_k(G)$ is normal, $N=M_1+\cdots+M_t$ where each $M_i$ is a $k$-matching of $G$. Since $N$ and the $M_i$'s have nonnegative entries, every time that $N$ has a zero in the $(r,s)$-entry so do the $M_i$'s. Thus the $M_i$'s are $k$-matchings of $H$, and therefore $\M_k(H)$ is also normal.
\end{proof}
Since every bipartite graph is a subgraph of the complete bipartite graph $\K_{n,n}$ for some $n\in\mathbb{Z}^+$, it is enough to prove our main result for the polytopes $\M_k(\K_{n,n})$, and that is what we are going to do.
\begin{definition}
Let $G$ be a bipartite graph with vertex sets $A=\{a_1,\cdots,a_m\}$ and $B=\{b_1,\cdots,b_n\}$. If $N\in t\M_k(G)$ ($t\in\mathbb{Z}^+$) is an integer point, then the graph $H$ induced by $N$ is the graph with the same vertices of $G$ and such that $\{a_i,b_j\}$ is an edge of $H$ if and only if the entry $(i,j)$ of $N$ is nonzero. 
\end{definition}
For completeness we present a proof that Birkhoff polytopes $B_n=\M_n(\K_{n,n})$ are normal (Theorem \ref{t1}). This is a known result in matching theory; nonetheless it is difficult to find a reference for its proof. We begin by extracting a lemma which we will use later on.
\begin{lemma}\label{lemma1matchings}
Let $N\in tB_n$ be a an integer point. Then the graph induced by $N$ has a perfect matching $M$, and thus $N-M\in(t-1)B_n$.  
\end{lemma}
\begin{proof}
Let $H$ be the graph induced by $N$, and suppose it has no perfect matching. By Hall's Marriage Theorem, there is a subset $S$ of $A$ such that $|\Gamma(S)|<|S|$, where $\Gamma(S)$ is the set of vertices in $B$ that are connected in $H$ by an edge to at least one element in $S$. Let $k=|S|$ and, without loss of generality, assume $S=\{a_1,\ldots,a_k\}$ and $B\smallsetminus\Gamma(S)=\{b_1,\ldots,b_m\}$ where $m=n-|\Gamma(S)|$. From $|\Gamma(S)|<|S|$, it follows that $m=n-|\Gamma(S)|>n-|S|$, so $m\geq n-k+1$. If $b_j\notin\Gamma(S)$, then the entries $N_{1j},\ldots,N_{kj}$ of $N$ are zero. Thus in the upper-left corner of $N$, there is a submatrix of zeros of size $k\times(n-k+1)$.

Since $N\in tB_n$, the entries of $N$ in each row and column have to sum up to $t$. This means that the entries of $N$ in the upper-right $k\times(k-1)$-submatrix of $N$ must sum up to $kt$, but on the other hand, the entries in the right $n\times(k-1)$ submatrix of $N$ must sum up to $(k-1)t$. This is a contradiction.       
\end{proof}
\begin{theorem}\label{t1}
The Birkhoff polytope $B_n$
is normal.
\end{theorem}
\begin{proof}
We prove by induction on $t$ that every integer point in $tB_n$ is the sum of $t$ integer points in $B_n$. The result clearly holds for $t=1$. So suppose it is true for $t-1$, and we will prove that it also holds for $t$.

Let $N$ be an integer point of $tB_n$. By Lemma \ref{lemma1matchings}, the graph induced by $N$ has a perfect matching $M$, and thus $N-M\in(t-1)B_n$ which by induction can be written as the sum of $t-1$ integer points of $B_n$. Notice that $M$ is an integer point of $B_n$, so $N$ is the sum of $t$ integer points of $B_n$.
\end{proof}
The plan now is to generalize both Lemma \ref{lemma1matchings} and Theorem \ref{t1} to the case of the polytopes $\M_k(\K_{n,n})$. For that, we need the next auxiliary result.
\begin{lemma}\label{l3}
Let $G$ be a bipartite graph on $A\sqcup B$. Given $A'\subseteq A$ with $|A'|=r$, $B'\subseteq B$ with $|B'|=c$, and a nonnegative integer $k\leq r+c$, suppose:
\begin{itemize}
\item there is a $k$-matching $M_1$ of $G$ covering $A'$
\item there is a $k$-matching $M_2$ of $G$ covering $B'$
\item there is a $(r+c-k)$-matching $M_3$  of the subgraph induced by $A'\sqcup B'$. 
\end{itemize}
Then $M_1\cup M_2\cup M_3$ contains a $k$-matching of $G$ covering $A'\sqcup B'$.
\end{lemma}
\begin{proof}
We prove this by induction on $r+c-k$. If $r+c-k=0$, we use Lemma \ref{l2}: if $V_1$ is the set of all vertices in $A$ that are covered by $M_1$, and $V_2$ is the set of all vertices in $B$ that are covered by $M_2$, then $M_1\cup M_2$ contains a matching $M$ that covers $V_1\cup V_2$. This matching $M$ contains at least $k$ edges since $|V_1|=|M_1|=k$, and it uses at most $k$ edges to cover $A'\sqcup B'$ since $k=r+c$. So we can remove edges from $M$ to obtain a $k$-matching covering $A'\sqcup B'$. Suppose now that $r+c-k>0$ and that the result holds for smaller values.

\textbf{Case 1}: $M_1$ or $M_2$ contains an edge from $A'$ to $B'$. Without loss of generality we assume the edge is in $M_1$. If $M_1$ contains the edge $e=\{a,b\}$ where $a\in A'$ and $b\in B'$, define $M'_1:=M_1\setminus\{e\}$, $M'_2:=M_2\setminus\{e'\}$ where $e'$ is the edge of $M_2$ incident to $b$, and $M'_3$ as the set obtained from $M_3$ by removing from this all the edges incident to $a$ and $b$; in case that there is no edge in $M_3$ incident to $a$ or $b$, remove any arbitrary edge of $M_3$. Thus $M'_3$ has one or two elements less than $M_3$.

If $|M_3\setminus M'_3|=1$, we have
\begin{itemize}
\item a matching $M'_1$ of size $k-1$ covering $A'\setminus\{a\}$
\item a matching $M'_2$ of size $k-1$ covering $B'\setminus\{b\}$
\item a matching $M'_3$ of size $r+c-k-1=(r-1)+(c-1)-(k-1)$ in the subgraph induced by $(A'\setminus\{a\})\sqcup(B'\setminus\{b\})$, 
\end{itemize}
so by the induction hypothesis $M'_1\cup M'_2\cup M'_3$ contains a matching of size $k-1$ covering $(A'\setminus\{a\})\sqcup(B'\setminus\{b\})$. Add $e$ to this matching, and the result follows.

If $|M_3\setminus M'_3|=2$, we have two edges in $M_3$ of the form $e_1=\{a,b'\}$ and $e_2=\{a',b\}$ with $a\neq a'\in A'$ and $b\neq b'\in B$. Remove from $M'_1$ the edge that is incident to $a'$ and from $M'_2$ the edge that is incident to $b'$. Call this new sets $M''_1$ and $M''_2$. We have then
\begin{itemize}
\item a matching $M''_1$ of size $k-2$ covering $A'\setminus\{a,a'\}$
\item a matching $M''_2$ of size $k-2$ covering $B'\setminus\{b,b'\}$
\item a matching $M'_3$ of size $r+c-k-2=(r-2)+(c-2)-(k-2)$ in the subgraph induced by $(A'\setminus\{a,a'\})\sqcup(B'\setminus\{b,b'\})$, 
\end{itemize}
so by the induction hypothesis $M''_1\cup M''_2\cup M'_3$ contains a matching of size $k-2$ covering $(A'\setminus\{a,a'\})\sqcup (B'\setminus\{b,b'\})$. Add $e_1$ and $e_2$ to this matching, and the result follows.

\textbf{Case 2}: Neither $M_1$ nor $M_2$ has an edge from $A'$ to $B'$. Let $A''$ be the vertices in $A'$ that are covered by $M_3$, and define $B''$ similarly. Let $M'_1:=\{e\in M_1: e \text{ is incident to a vertex in } A'\setminus A''\}$, and define $M'_2$ similarly. Notice that $|M'_1|=|A'|-|A''|=r-(r+c-k)=k-c$, and similarly, $|M'_2|=k-r$. Then $M'_1\sqcup M'_2\sqcup M_3$ is a matching of $G$ of size $|M'_1|+|M'_2|+|M_3|=(k-c)+(k-r)+(r+c-k)=k$ covering $A'\sqcup B'$ (no induction is required in this case). 
\end{proof}
To make it easier to understand the following proofs, let us consider the elements in $\M_k(\K_{n,n})$ and $t\M_k(\K_{n,n})$ as matrices. By Theorem \ref{h-representation},
\begin{align*}
\M_k(\K_{n,n}):=\{X\in\mathbb{R}^{n\times n}: &\sum_{j=1}^{n}x_{ij}\leq 1\  \ \forall i\in [n],\ \sum_{i=1}^{n}x_{ij}\leq 1\  \ \forall j\in [n],\\
& x_{ij}\geq 0\  \ \forall i,j\in [n],\\
&\sum_{i,j}x_{ij}=k\}.
\end{align*}
Since $X\in t\M_k(\K_{n,n})$ if and only if $\frac{1}{t}X\in\M_k(\K_{n,n})$, we have
\begin{align*}
t\M_k(\K_{n,n}):=\{X\in\mathbb{R}^{n\times n}: &\sum_{j=1}^{n}x_{ij}\leq t\  \ \forall i\in [n],\ \sum_{i=1}^{n}x_{ij}\leq t\  \ \forall j\in [n],\\
& x_{ij}\geq 0\  \ \forall i,j\in [n],\\
&\sum_{i,j}x_{ij}=tk\}.
\end{align*} 
\begin{lemma}\label{l1}
Let $N\in t\M_k(\K_{n,n})$ be an integer point. Then the graph induced by $N$ has a $k$-matching $M$ such that $N-M\in(t-1)\M_k(\K_{n,n})$.
\end{lemma}
\begin{proof}
The case $k=n$ is Lemma \ref{lemma1matchings}. We prove this result by induction on $k$. The case $k=0$ is trivial, so suppose $0<k<n$, and that the result is true for smaller values.

Since $k<n$, there is a row $i$ such that the sum of its entries is less than $t$, since otherwise the sum of all the entries in $N$ would be $tn$. Similarly there is a column $j$ such that the sum of its entries is less than $t$. We add 1 to the entry $(i,j)$ to obtain a new matrix $N_1$. We now repeat this process to $N_1$ to obtain a new matrix $N_2$. Repeating this process $t(n-k)$ times, we obtain a matrix $N_{t(n-k)}\in B_n$. By the normality of $B_n$, $N_{t(n-k)}=M^*_1+\cdots+M^*_t$ where each $M^*_i$ is a perfect matching of $K_{n,n}$.

Think of the 1's we added to $N$ to obtain $N_{t(n-k)}$ as colored with red, and think of each original entry of $N$ as the sum of black 1's. Since $N_{t(n-k)}=M^*_1+\cdots+M^*_t$, each 1 in $N_{t(n-k)}$ (black or red) has to appear in one of the $M^*_i$'s. The total number of black 1's that appear in all the $M^*_i$'s is $tk$, so the average number of black 1's in each perfect matching is $tk/t=k$. This implies that among the $M^*_i$'s there is at least one, which we denote by $M^*$, with at least $k$ black 1's. This implies that there is a $k$-matching $M$ of $H$, where $H$ is the graph induced by $N$. The problem is that $N-M$ does not necessarily belong to $(t-1)\M_k(\K_{n,n})$ since it may have rows or columns that sum up to $t$, so further analysis has to be done.

Since we never add red 1's in those rows and columns whose entries sum up to $t$, we have a black 1 in each of these rows and columns in $M^*$. Let $r$ and $c$ be the number of rows and columns of $N$, respectively, whose entries sum up to $t$. If $r+c\leq k$, then we can always take $k$ black 1's from $M^*$, to form a $k$-matching $M$ of $H$, in such a way that we take all 1's in those rows and columns whose entries sum up to $t$. This implies that $N-M\in(t-1)\M_k(\K_{n,n})$.

Let $A$ and $B$ be the sets which partition the vertices of $\K_{n,n}$. Since $r,c\leq k$ (otherwise the sum of the entries of $N$ would be greater than $tk$), we can always find a $k$-matching $M_1$ that covers
\[A':=\{\text{vertices in } A \text{ whose respective rows sum up to } t\}\]
and a $k$-matching $M_2$ that covers
\[B':=\{\text{vertices in } B \text{ whose respective columns sum up to } t\}.\]
Let's consider now the case $r+c>k$. Without loss of generality suppose that the first $r$ rows and the first $c$ columns of $N$ are the ones whose entries sum up to $t$. Then $N$ can be written in blocks as
\[
\begin{LARGE}
\begin{bmatrix}
A_1 & A_2 \\
A_3 & A_4 \\
\end{bmatrix} 
\end{LARGE} 
\]
where $A_1$ is a block of size $r\times c$. If $a_i$ denotes the sum of the entries in $A_i$, then $a_1+a_2+a_3\leq tk$. Thus $tr+tc=(a_1+a_2)+(a_1+a_3)\leq a_1+tk$, and therefore $t(r+c-k)\leq a_1$.

If $r=c=k$, $A_2$, $A_3$, and $A_4$ would be blocks with only zeros, and $A_1\in tB_k$. By Lemma \ref{lemma1matchings}, the graph induced by $A_1$ has a perfect matching $P$, and in this case the $k$-matching 
\[M =
\begin{bmatrix} P & \textbf{0}\\
\textbf{0} & \textbf{0}\end{bmatrix}
\]
satisfies that $N-M\in(t-1)\M_k(\K_{n,n})$.

So suppose now that $r$ or $c$ is not equal to $k$, and thus $0<r+c-k<k$.

We reduce some positive entries of $A_1$ to obtain a matrix $C$ such that its entries sum up to $t(r+c-k)$, which we can do since $t(r+c-k)\leq a_1$. Then
\[
\begin{bmatrix} C & \textbf{0}\\ \textbf{0} & \textbf{0} \end{bmatrix}\in t\M_{r+c-k}(\K_{n,n}).
\]
Since $0<r+c-k<k$, by induction, the graph induced by
\[
\begin{bmatrix} C & \textbf{0}\\ \textbf{0} & \textbf{0} \end{bmatrix}
\]
has a matching of size $r+c-k$ of the form
\[
\begin{bmatrix} D & \textbf{0}\\ \textbf{0} & \textbf{0} \end{bmatrix}.
\]
From the way $C$ is obtained, we have that the induced graph of
\[
\begin{bmatrix} C & \textbf{0}\\ \textbf{0} & \textbf{0} \end{bmatrix}
\]
is a subgraph of the induced graph $H$ of $N$. Thus
\[
\begin{bmatrix} D & \textbf{0}\\ \textbf{0} & \textbf{0} \end{bmatrix}
\]
is also a matching of $H$ of size $r+c-k$. That is, we have found a matching $M_3$ of size $r+c-k$ in the subgraph induced by $A'\sqcup B'$. We use Lemma \ref{l3} to conclude that there is a matching $M$ of $H$ of size $k$ covering $A'\sqcup B'$, which implies that $N-M\in(t-1)\M_k(\K_{n,n})$.
\end{proof}
\begin{proposition}\label{propnormal}
The polytope $\M_k(\K_{n,n})$ is normal.
\end{proposition}
\begin{proof}
We need to prove that if $N\in t\M_k(\K_{n,n})$, then it can be written as the sum of $t$ integer points of $\M_k(\K_{n,n})$, that is, as the sum of $t$ $k$-matchings of $\K_{n,n}$. We prove this by induction on $t$.

The case $t=1$ is clear. Suppose the result is true for $t-1$, and let us prove it for $t$. By Lemma \ref{l1}, there is a $k$-matching of the induced graph of $N$, which is also a $k$-matching of $\K_{n,n}$, such that $N-M\in(t-1)\M_k(\K_{n,n})$. By the induction hypothesis, $N-M$ is the sum of $t-1$ $k$-matchings of $\K_{n,n}$. Thus $N$ is the sum of $t$ $k$-matchings of $\K_{n,n}$.
\end{proof}
\begin{theorem}
Let $G$ be a bipartite graph and $k\in\mathbb{N}$. Then the polytope $\M_k(G)$ is normal.
\end{theorem}
\begin{proof}
As we remarked before, every bipartite graph can be seen as a subgraph of $\K_{n,n}$ for some $n\in\mathbb{Z}^+$. Thus, our result follows from Lemma \ref{lemmasubgraph} and Proposition \ref{propnormal}.
\end{proof}

\bibliographystyle{plain}
\bibliography{references}
\end{document}